\documentclass[a4paper,11pt,reqno, oneside]{amsart}

\usepackage[english]{babel}
\usepackage[latin1]{inputenc}

\usepackage[verbose,tmargin=1in,bmargin=1in,lmargin=1in,rmargin=1in]{geometry}
\usepackage{amsfonts}
\usepackage{amsmath}
\usepackage{amssymb}
\usepackage{amsthm}
\usepackage{graphicx}
\usepackage{caption}
\usepackage{color}
\usepackage[unicode=true,pdfusetitle,
    bookmarks=true,bookmarksnumbered=false,bookmarksopen=false,
breaklinks=true,pdfborder={0 0 0},backref=false,colorlinks=true]{hyperref}
\hypersetup{linkcolor=blue,citecolor=blue,urlcolor=blue}

\usepackage{todonotes}
\usepackage{enumerate}
\usepackage{ulem}

\newtheorem{theorem}{Theorem}

\newtheorem{lemma}{Lemma}

\begin{document}

 \author{G. Bauer}
 \email{gernot.bauer@fh-muenster.de}
\address{Fachhochschule Münster, Soester Str. 13, 48155
M\"unster, Germany.}
 
\author{D.-A.\ Deckert}
\email{deckert@math.lmu.de}
\address{Mathematisches Institut der LMU, Theresienstr. 39, 80333
M\"unchen, Germany.}

\author{D.\ D\"urr}
\email{duerr@math.lmu.de}
\address{Mathematisches Institut der LMU, Theresienstr. 39, 80333
M\"unchen, Germany.}

\author{G.\ Hinrichs}
\email{hinrichs@math.lmu.de}
\address{Mathematisches Institut der LMU, Theresienstr. 39, 80333
M\"unchen, Germany.}

\title[Global solutions to the FST equations on a straight line]{Global solutions to the\\ electrodynamic two-body problem\\ on a
straight line}

\begin{abstract}
    The classical electrodynamic two-body problem has been a long standing open
    problem in mathematics. For motion constrained to the straight line, the
    interaction is similar to that of the two-body problem of classical
    gravitation. The additional complication is the presence of unbounded
    state-dependent delays in the Coulomb forces due to the finiteness of the
    speed of light.  This circumstance renders the notion of local solutions
    meaningless, and therefore, straight-forward ODE techniques can not be
    applied. Here, we study the time-symmetric case, i.e., the
    Fokker-Schwarzschild-Tetrode (FST) equations, comprising both advanced and
    retarded delays.  We extend the technique developed in \cite{DirkGuenter},
    where existence of FST solutions was proven on the half-line, to ensure
    global existence -- a result that had been obtained by Bauer \cite{Bauer} in
    1997.  Due to the novel technique, the presented proof is shorter and more
    transparent but also relies on the idea to employ asymptotic data to
    characterize solutions.\\

    \textbf{Keywords:} Fokker-Schwarzschild-Tetrode electrodynamics;
    Wheeler-Feynman electrodynamics; delay differential equations

\end{abstract}

\maketitle

\section{Introduction}

The time-symmetric electrodynamic interaction of point-charges is described by
the so-called Fokker-Schwarzschild-Tetrode (FST) equations. Historically, these
equations were first discussed in the works \cite{Schwarzschild, Tetrode,
Fokker}. Later Wheeler and Feynman took up these equations in their seminal
works \cite{WF1,WF2} to show that the electrodynamic arrow of time is derived
from the thermodynamic one; see also \cite{GernotDirkDetlefGuenter}. In our
case, in which we restrict ourselves to two point-charges moving along a
straight line having positions $a(t),b(t)\in\mathbb R$ at time
$t\in\mathbb R$, the FST equations take the form
\begin{equation}\label{WF}
 \begin{split} 
     \frac{\text d}{\text dt}\left(\frac{\dot a(t)}{\sqrt{1-\dot
     a(t)^2}}\right) &= \frac{\kappa_a}2\left[\frac{1+\dot
     b\left(t_2^-\right)}{1-\dot
     b\left(t_2^-\right)}\frac1{\left(a(t)-b(t_2^-)\right)^2} + \frac{1-\dot
     b\left(t_2^+\right)}{1+\dot
     b\left(t_2^+\right)}\frac1{\left(a(t)-b(t_2^+)\right)^2}\right]\,, \\
     \frac{\text d}{\text dt}\left(\frac{\dot b(t)}{\sqrt{1-\dot
     b(t)^2}}\right) &= -\frac{\kappa_b}2\left[\frac{1-\dot
     a\left(t_1^-\right)}{1+\dot
     a\left(t_1^-\right)}\frac1{\left(b(t)-a(t_1^-)\right)^2} +\frac{1+\dot
     a\left(t_1^+\right)}{1-\dot
     a\left(t_1^+\right)}\frac1{\left(b(t)-a(t_1^+)\right)^2}\right]\,.
 \end{split}
\end{equation}
Here, we use the dot notation, i.e., any derivative w.r.t.\ time parameter $t$
is denoted by an overset dot such as $\dot a(t)=\frac{d}{dt}a(t)$ and $\ddot a(t) =
\frac{d^2}{dt^2} a(t)$, and furthermore,
units such that speed of light equals one. Furthermore,
$\kappa_{a},\kappa_{b}>0$ denote coupling constants and the so-called advanced and retarded times $t^+_i$ and $t^-_i$ for $i=1,2$
are given implicitly as solutions to the following equations
\begin{align}\label{t+-}
    \begin{split}
        t^\pm_1&=t_1^\pm(a,b(t),t)=t\pm|a(t_1^\pm(a,b(t),t))-b(t)|\,,\\
        t^\pm_2&=t_2^\pm(a(t),b,t)=t\pm|a(t)-b(t_2^\pm(a(t),b,t))|\,.
    \end{split}
\end{align}
Thus, $t^\pm_1$ and $t_2^\pm$ are functionals of trajectory $a$ and $b$,
respectively. To keep the notation slim we will often omit their arguments.
Geometrically, equations \eqref{t+-} can be understood as the intersection
times of the forward and backward light-cones of the respectively other
trajectory. These intersection points exist as long as the
trajectories $a$ and $b$ have velocities that are bounded away from one, i.e., the speed of light.

We shall establish in this paper the existence of solutions to the FST equations
\eqref{WF} satisfying a prescribed asymptotic behavior; see
Theorem~\ref{mainresult} below.

The presence of advanced and delayed terms is the main source of
difficulty when aiming at a global existence result. Namely, it implies that the
right-hand sides of the equations of motion in \eqref{WF} involve not only terms
evaluated at the same time instant $t$ but also at the respectively future or
past times $t^\pm_i$. Since these times are potentially unbounded
functionals of the entire trajectories, the notion of local solutions is
meaningless as even a very small interval of the trajectory $a$ may depend on a
very large interval of the trajectory $b$ and vice versa. Therefore,
straight-forward ODE techniques based on finding local solutions first and
extending them to global solutions with an additional a priori estimate can not
be applied.  A particularly interesting question is therefore in which sense one
may hope for a well-posed initial value problem. Due to the delayed terms it is
not clear if in general Cauchy data, i.e., position and velocities of both
charges at one time instant, suffices to characterize solutions uniquely  or if
even whole strips of the trajectories have to be prescribed as initial data. 
In case the advanced terms are omitted, the respective equations are called
Synge equations. For this case, solutions on the half-line can be found by
integration,
but obtaining global
ones is still highly non-trivial; see
\cite{Driver1,Angelov,DirkGuenter}. The goal of this work, i.e.,
Theorem~\ref{mainresult} below, is to ensure global existence of solutions to
equations \eqref{WF} including both the advanced and retarded terms.  Before
Theorem~\ref{mainresult} can be spelled out precisely and its special
formulation can be understand we need to recall some previous results
and discuss the asymptotic behavior of potential solutions. 

To date the only result about uniqueness of solutions to the FST equations is
given in \cite{Driver2}. There, uniqueness of solutions for two charges on the
straight line was obtained in the special situation of initially prescribed zero
velocities and sufficiently large separation of the two charges; for a
general discussion of valid initial data and uniqueness in the case of a FST toy
model see \cite{DirkNicola}.  General global existence of solutions on the
straight line was later shown in \cite{Bauer}. As yet the only result towards a
solution theory of the FST equations in three space-dimensions is given in
\cite{Dirk2} where, for $N$ rigid charges and a prescribed asymptotic behavior
of the trajectories for times $|t|>\tau$, existence of solutions on $[-\tau,\tau]$ for
arbitrary large $0<\tau<\infty$ was shown.  Apart from the different setting, the
crucial difference in the latter two works lies in the prescription of the
initial data.  In \cite{Bauer} it was given as asymptotic data in the remote
past and in \cite{Dirk2} as Cauchy data. Both choices seem to have advantages
and disadvantages: In order to use asymptotic data, a priori, one must determine
the asymptotic behavior of potential scattering solutions; until now this was
only done successfully on the straight line; see \cite{Bauer}. This knowledge
then provides sufficient global control on the charge trajectories to employ
topological fixed-point methods. However, in general, such methods do not provide
information about uniqueness and about how the asymptotic data relates to
potential data at finite times.  On the contrary, possible notions of initial
data at finite times are suggested readily when recasting the FST equations into
integral form to yield potential candidates for self-maps. The trouble with this
approach is that one usually lacks sufficient control on the global behavior of
the solutions in order to apply fixed-point methods globally. Because of this one is
usually only able to prove existence of solutions on finite time intervals as in
\cite{Dirk2}.

In a recent work \cite{DirkGuenter}, for two charges on the straight line, it
was possible to infer from prescribed initial data existence of solutions not
only in a finite time interval but on the half-line.  More precisely, it
was shown in  \cite[Theorem 1.1]{DirkGuenter}:
\begin{theorem}[FST solutions on the half-line]
    \label{conditional}
    Given an initial position $a_0\in\mathbb R$ and velocity $\dot a_0\in]-1,1[$
    of charge $a$ at a time $T\in\mathbb R$, and in addition, an initial
    trajectory strip $b_0\in C^\infty\left([T^-,T^+],]-\infty,a_0[\right)$
    of charge $b$ such that $T^\pm=T\pm\left(a_0-b_0(T^\pm)\right)$, there
is a trajectory pair $(a,b)$ fulfilling:
\begin{enumerate}[(i)]
    \item Trajectory $a$ fulfills the equation of motion \eqref{WF} for all
        $t\ge T$ and $b$ for all $t\ge T^+$;
    \item The pair $(a,b)$ satisfies the initial conditions
        $a(T)=a_0,\dot a(T)=\dot a_0,\left.b\right|_{[T^-,T^+]}=b_0$;
    \item The pair $(a,b)$ 
        fulfills $a\in C^\infty([T,\infty[)$ and $b\in C^\infty([T^+,\infty[)$.
\end{enumerate}
\end{theorem}
Henceforth, we will refer to such solutions as ``conditional solutions''
corresponding to the prescribed initial data. This result is basic to
this work: In
order to prove the global existence result for the FST equations, Theorem~\ref{mainresult} below, we
combine Theorem~\ref{conditional} with the a priori results on the asymptotic
behavior of solutions that were also exploited in \cite{Bauer}.  The strategy of
proof consists of two steps: First, we identify the asymptotic trajectories
$(x,y)$ of potential global solutions. As it turns out, the velocities $(\dot
a,\dot b)$ of any
global solution $(a,b)$ are bounded away from one 
\cite[(24a) in Proposition 2.1]{DirkGuenter}, and since
$\ddot a>0$ and $\ddot b<0$ due to \eqref{WF}, they converge to limiting values
\begin{align}
    \label{asymp-vels}
    -1<u_{-\infty}=\lim_{t\to-\infty}\dot a(t)\quad<\quad\lim_{t\to-\infty}\dot
    b(t)=v_{-\infty}<1\,.
\end{align}
Moreover, as we discuss in
Section~\ref{asymptotes}, it
is even possible to show that these two asymptotic velocities together with two
other reals $x_{-\infty},y_{-\infty}$, referred to as asymptotic positions,
parametrize all possible asymptotes $(x,y)$ of potential global solutions. In a
second step, relying on this asymptotic information, we will then extract
initial data to infer conditional solutions conditional solutions $a_T,b_T$
satisfying
\begin{equation}\label{Anfangsbed}
a_T(T)=x(T)\,,\, \left.b_T\right|_{[T^-,T^+]} = \left.y\right|_{[T^-,T^+]}\,,
\end{equation}
for sufficiently large negative number $T$ (meaning large magnitude $|T|$ but
$T<0$), where
\begin{equation}\label{T+-}
 T^\pm=t_2^\pm(x(T),y,T)\,.
\end{equation} 
For convenience, $a_T$ and $b_T$ will be extended to all times $t\in\mathbb R$
by means of the asymptotes $(x,y)$
\begin{align}
    \label{ext}
    a_T(t)=x(t) \text{ for } t\in ]-\infty,T[
    \qquad
    b_T(t)=y(t) \text{ for } t\in ]-\infty,T^-[\,.
\end{align}
  Since
their velocities are bounded away from the speed of light uniformly in $T$ we
are able to establish sufficiently strong uniform estimates that allow to prove
that a certain subsequence of the family of trajectory pairs $(a_T,b_T)$
converges in a suitable sense to a global solution to \eqref{WF} as
$T\to\infty$.\\

The paper is structured as follows: In Section~\ref{asymptotes} we 
discuss the asymptotic behavior of potential global solutions. In
Section~\ref{central} we use the introduced notions to characterize the
asymptotes in order to formulate our global existence result in
Theorem~\ref{mainresult} and provide the proof.

\section{Asymptotic behavior of potential global solutions}
\label{asymptotes}

As the highest derivative in \eqref{WF} is on the right-hand side, one
immediately observes that any global solution $(a,b)$ to \eqref{WF} must be
smooth. As discussed earlier \eqref{asymp-vels}, the velocities $\dot a$ and
$\dot b$ of any potential global solution have a modulus which is bounded away
from one, which implies that the implicitly defined functions $t_i^\pm$ in
\eqref{t+-} are well-defined and smooth.  Moreover, because of $\ddot a>0$ and
$\ddot b<0$, the velocities  converge monotonically to asymptotic velocities
\eqref{asymp-vels}.  Without restricting generality, we will assume
\begin{align*}
   a(t)>b(t)
   \qquad
   \text{for }t\in\mathbb R
\end{align*} 
throughout the paper.  In other words, asymptotically, the
velocities are constant, and to a first approximation, the difference
$a(t)-b(t_2^\pm(t))$ can be expected to be of order $t$ for large $|t|$.  The
other terms on the right-hand side in \eqref{WF} involve only velocities, so
that the acceleration
\begin{equation}\label{Beschleunigung}
\ddot a(t)=(1-\dot a(t)^2)^{\frac32}\frac{\text d}{\text dt}\left(\frac{\dot
a(t)}{\sqrt{1-\dot a(t)^2}}\right)
\end{equation}
should be of order $\frac1{t^2}$ for large $|t|$, which leads to the
guess
\begin{equation}\label{Ansatz}
a(t)=x_{-\infty}+u_{-\infty}t-C\ln(-t)+o\left(\ln(-t)\right)
\end{equation}
for $t\to-\infty$, where we use the notation $f(s)=o(g(s))\Leftrightarrow
\lim_{s}\frac{f(s)}{g(s)}=0$ and $f(s)\sim g(s)\Leftrightarrow
\lim_{s}\frac{f(s)}{g(s)}=1$. The logarithmic corrections in
\eqref{Ansatz} are characteristic to Coulomb
interaction in 3+1 dimensions and well-known in the case without delays (note
that although we restrict the dynamics to a straight line, we use the
electrodynamic interaction, which is determined by the Green's function of the
d'Alembert operator, in 3+1 dimensions). 

The result \eqref{Ansatz} was obtained in \cite{Bauer} for the first time,
but since that publication is only available in German and it is crucial for
understanding the formulation of our main result, Theorem~\ref{mainresult}, we
shall briefly discuss the intuition behind its proof here. This discussion will
consist of a rigorous part until equation \eqref{Ansatz2} which will provide
formulas that will be used in the proof of our main result and, in particular,
make the reader familiar with the terms \eqref{A1}-\eqref{A5} below. The shorter
second part will then explain in a nonrigorous way how the estimate \eqref{Ansatz2}
leads to \eqref{Ansatz} including an identification of the constant
$C$. However, a proof of the latter will not be given in this section. Rather,
it will proven as part of the main result Theorem~\ref{mainresult} that the
established solutions comply with the asymptotic behavior given here.

In order to infer more details about the asymptotic behavior, it is convenient
to bring the integrated version of the equation of
motion \eqref{WF}, i.e.,
\begin{equation}\label{WFintTt}
 \begin{split}
 \dot a(t)
 =&\dot a(T)+\frac{\kappa_a}{2}\int_T^t(1-\dot a(s)^2)^{\frac32} \\
 & \hskip2cm \times \left[\frac{1+\dot b\left(t_2^-\right)}{1-\dot b\left(t_2^-\right)}\frac1{\left(a(s)-b(t_2^-)\right)^2} 
 + \frac{1-\dot b\left(t_2^+\right)}{1+\dot b\left(t_2^+\right)}\frac1{\left(a(s)-b(t_2^+)\right)^2}\right]\text ds \,,
 \end{split}
\end{equation}
for $t,T\in\mathbb R$,
into a form similar to \eqref{Ansatz} by partial integration such that the
factors \\$\left(a(s)-b(t_2^\pm(t))\right)^{-2}$ turn into 
$\ln\left(a(s)-b(t_2^\pm(t))\right)$. By definition of $t^\pm_i$
in \eqref{t+-}, we have
\begin{equation}\label{tpunkt}
 \dot t_2^\pm(s)=1\pm\dot a(s)\mp\dot b(t_2^\pm(s))\dot t_2^\pm(s)=\frac{1\pm
 a(s)}{1\pm\dot b(t_2^\pm(s))}\,,
\end{equation}
and therefore,
\begin{equation}\label{1-tpunkt}
 \frac{\text d}{\text ds}\left(a(s)-b(t_2^\pm(s))\right)=\pm\frac{\text d}{\text
 ds}\left(t_2^\pm(s)-s\right)=\pm(\dot t_2^\pm(s)-1)=\frac{\dot a(s)-\dot
 b(t_2^\pm(s)}{1\pm\dot b(t_2^\pm(s))}\,.
\end{equation}
We recall from \cite[Lemma 2.1]{DirkGuenter} that the modulus in definition
\eqref{t+-} can be omitted since, as
the speed of light equals one, one has
\begin{align}
    \label{eq:t+-est}
    \frac{a(t)-b(t)}{2}
    \leq
    \left.\begin{cases}
    a(t)-b(t^\pm_2)\\ 
    a(t^\pm_1)-b(t)
    \end{cases}\right\rbrace
    \leq
    \frac{a(t)-b(t)}{1-\max\{\|\dot a\|_\infty,\|\dot b\|_\infty\}}\,,
\end{align}
with $\|\cdot\|_\infty$ denoting the $L^\infty$ norm.
Consequently, in the first integration by parts, we take the
antiderivative of
\begin{align*}
    \frac{d}{ds}\frac1{a(s)-b(t_2^\pm(s))}=-\frac{\dot a(s)-\dot b(t_2^\pm(s)}{1\pm\dot
    b(t_2^\pm(s))}\frac1{\left(a(s)-b(t_2^\pm(s))\right)^2}
\end{align*}
and the derivative of
the rest. This yields
\allowdisplaybreaks
\begin{subequations}
\begin{align}
    \dot a(t) =& \dot a(T) \\
  &+ \frac{\kappa_a}2\left(1-\dot a(T)^2\right)^{\frac32}\left[\frac{1+\dot b(t_2^-(T))}{\dot a(T)-\dot b(t_2^-(T))}\frac1{a(T)-b(t_2^-(T))}\right. \\
  &\hphantom{\frac{\kappa_a}2\left(1-\dot a(T)^2\right)^{\frac32}\left[\right.}+\left.\frac{1-\dot b(t_2^+(T))}{\dot a(T)-\dot b(t_2^+(T))}\frac1{a(T)-b(t_2^+(T))}\right] \\
  &- \frac{\kappa_a}2\left(1-\dot a(t)^2\right)^{\frac32}\left[\frac{1+\dot
  b(t_2^-(t))}{\dot a(t)-\dot b(t_2^-(t))}\frac1{a(t)-b(t_2^-(t))}\right. 
  \label{stepone}\\
  &\hphantom{\frac{\kappa_a}2\left(1-\dot a(t)^2\right)^{\frac32}\left[\right.} + \left.\frac{1-\dot b(t_2^+(t))}{\dot a(t)-\dot b(t_2^+(t))}\frac1{a(t)-b(t_2^+(t))}\right] \\
  &+ \frac{\kappa_a}2\int_T^t \left[\frac{\text d}{\text ds}\left(\left(1-\dot a(s)^2\right)^{\frac32}\frac{1+\dot b(t_2^-(s))}{\dot a(s)-\dot b(t_2^-(s))}\right)\right]\frac1{a(s)-b(t_2^-(s))}\text ds \\
  &+ \frac{\kappa_a}2\int_T^t \left[\frac{\text d}{\text ds}\left(\left(1-\dot a(s)^2\right)^{\frac32}\frac{1-\dot b(t_2^+(s))}{\dot a(s)-\dot b(t_2^+(s))}\right)\right]\frac1{a(s)-b(t_2^+(s))}\text ds\,.
\end{align}
\end{subequations}
In the second step, we integrate the latter equation for $\dot a(t)$ again from
$T$ to $t$ in order to obtain an equation for $a(t)$. The corresponding integral
over \eqref{stepone} is again done by parts using
$\frac{d}{ds}\ln(a(s)-b(t_2^\pm(s)))=\frac{\dot a(s)-\dot
b(t_2^\pm(s))}{1\pm\dot b(t_2^\pm(s))}\frac1{a(s)-b(t_2^\pm(s))}$ as antiderivate. We find
\allowdisplaybreaks
\begin{subequations}\label{WFint}
\begin{align}
 a(t)=& a(T)+\dot a(T)(t-T) \label{WFinta}\\
 &+ \frac{\kappa_a}2\left(1-\dot a(T)^2\right)^{\frac32}\left[\frac{1+\dot b(t_2^-(T))}{\dot a(T)-\dot b(t_2^-(T))}\frac{T}{a(T)-b(t_2^-(T))}\right. \nonumber\\
 &\hphantom{\frac{\kappa_a}2\left(1-\dot a(T)^2\right)^{\frac32}\left[\right.}  + \left.\frac{1-\dot b(t_2^+(T))}{\dot a(T)-\dot b(t_2^+(T))}\frac{T}{a(T)-b(t_2^+(T))}\right] \frac{t-T}T \label{WFintb}\\
 &- \frac{\kappa_a}2\left(1-\dot a(t)^2\right)^{\frac32}\left[\frac{1-\dot b(t_2^-(t))^2}{\left(\dot a(t)-\dot b(t_2^-(t))\right)^2}\ln\left(a(t)-b(t_2^-(t))\right)\right. \nonumber\\
 &\hphantom{\frac{\kappa_a}2\left(1-\dot a(t)^2\right)^{\frac32}\left[\right.} + \left. \frac{1-\dot b(t_2^+(t))^2}{\left(\dot a(t)-\dot b(t_2^+(t))\right)^2}\ln\left(a(t)-b(t_2^+(t))\right)\right] \label{WFintc}\\
 &+ \frac{\kappa_a}2\left(1-\dot a(T)^2\right)^{\frac32}\left[\frac{1-\dot b(t_2^-(T))^2}{\left(\dot a(T)-\dot b(t_2^-(T))\right)^2}\ln\left(a(T)-b(t_2^-(T))\right)\right. \nonumber\\
 &\hphantom{\frac{\kappa_a}2\left(1-\dot a(T)^2\right)^{\frac32}\left[\right.} + \left. \frac{1-\dot b(t_2^+(T))^2}{\left(\dot a(T)-\dot b(t_2^+(T))\right)^2}\ln\left(a(T)-b(t_2^+(T))\right)\right] \label{WFintd}\\
 &+ \frac{\kappa_a}2\int_T^t \left[\frac{\text d}{\text ds}\left(\left(1-\dot a(s)^2\right)^{\frac32}\frac{1-\dot b(t_2^-(s))^2}{\left(\dot a(s)-\dot b(t_2^-(s))\right)^2}\right)\right]\ln\left(a(s)-b(t_2^-(s))\right) \nonumber\\
  &\hphantom{\frac{\kappa_a}2\int_T^t} + \left[\frac{\text d}{\text ds}\left(\left(1-\dot a(s)^2\right)^{\frac32}\frac{1-\dot b(t_2^+(s))^2}{\left(\dot a(s)-\dot b(t_2^+(s))\right)}\right)\right]\ln\left(a(s)-b(t_2^+(s))\right)\text ds \label{WFinte}\\
 &+ \frac{\kappa_a}2\int_T^t \int_T^s \left[\frac{\text d}{\text dr}\left(\left(1-\dot a(r)^2\right)^{\frac32}\frac{1+\dot b(t_2^-(r))}{\dot a(r)-\dot b(t_2^-(r))}\right)\right]\frac1{a(r)-b(t_2^-(r))} \nonumber\\
  &\hphantom{\frac{\kappa_a}2\int_T^t\int_T^s} + \left[\frac{\text d}{\text dr}\left(\left(1-\dot a(r)^2\right)^{\frac32}\frac{1-\dot b(t_2^+(r))}{\dot a(r)-\dot b(t_2^+(r))}\right)\right]\frac1{a(r)-b(t_2^+(r))} \text dr\text ds\,.\label{WFintf}
\end{align}
\end{subequations}
Let us turn to the asymptotic behavior where we are interested in negative $t$
and have the arbitrary parameter $T$ at our expense. Thanks to the convergence
of the velocities in \eqref{asymp-vels},
identity \eqref{1-tpunkt}, and the de l'Hospital rule, we observe
$a(t)-b(t_2^\pm(t))\sim\frac{u_{-\infty}-v_{-\infty}}{1-v_{-\infty}}t$ for
$t\to-\infty$
so that
\begin{equation}\label{eta1}
 \eta_1:=
 \lim_{t\to-\infty}
 \kappa_a(1-\dot a(t)^2)^{\frac32}\left[\frac{1-\dot b(t^\pm_2)^2}{(\dot a(t)-\dot b(t^\pm_2(t)))^2}\right]
 =\kappa_a(1-u_{-\infty}^2)^{\frac32}\left[\frac{1-v_{-\infty}^2}{(u_{-\infty}-v_{-\infty})^2}\right]\,,
\end{equation}
and, $(\ref{WFintb})\sim \eta_1 \frac{t-T}{T}$ for $T\to-\infty$.
Furthermore, we have
$$\ln\left(a(t)-b(t_2^\pm(t))\right)\sim\ln\left(-\frac{u_{-\infty}-v_{-\infty}}{1\pm
    v_{-\infty}}\right)+\ln(-t)\,,$$
for $t\to-\infty$, so that 
\begin{equation*}
\begin{split}
(\ref{WFintc})
&\sim-\frac{\kappa_a}2(1-u_{-\infty}^2)^{\frac32}\frac{1-v_{-\infty}^2}{(u_{-\infty}-v_{-\infty})^2}\\
    & \hskip2cm\times\Big[2\ln(-t)+ \ln\left(-\frac{u_{-\infty}-v_{-\infty}}{1- v_{-\infty}}\right)
+ \ln\left(-\frac{u_{-\infty}-v_{-\infty}}{1+ v_{-\infty}}\right) \Big] \\
&= -\eta_1\ln(-t)-\frac{\eta_1}2\ln\left(-\frac{(u_{-\infty}-v_{-\infty})^2}{1-v_{-\infty}^2}\right)
\end{split}
\end{equation*}
for $t\to-\infty$ and, correspondingly,
$$(\ref{WFintd})\sim\eta_1\ln(-T)+\frac{\eta_1}2\ln\left(-\frac{(u_{-\infty}-v_{-\infty})^2}{1-v_{-\infty}^2}\right)$$
for $T\to-\infty$. When computing the derivatives in the remaining integral and
double integral terms by chain rule, each summand is proportional to $\ddot
a(t)$ or $\ddot b(t)$. Since $T$ is at our expense, we choose a scaling such
that $T\ll t\ll 0$. In this case, thanks to the equation of motion \eqref{WF}, for each
acceleration one gains again a $(a(t)-b(t^\pm_2))^{-2}$ or $(a(t^\pm_1)-b(t))^{-2}$ term, and
we find that \eqref{WFinte} is of order $\frac{\ln(-t)}{t}$ and
\eqref{WFintf} of order $\frac1t$ and, as such, subleading w.r.t.\ the
logarithmic correction. This implies
\begin{equation}\label{Ansatz2}
a(t)=a(T)+\dot
a(T)(t-T)+\eta_1\frac{t-T}{T}-\eta_1\ln(-t)+\eta_1\ln(-T)+o(\ln(-t))
\end{equation}
for $t,T\to-\infty$.

Next, in the second part, we discuss nonrigorously how \eqref{Ansatz2} leads
to \eqref{Ansatz} and how to identify the constant $C$. The proof of our main
result, Theorem~\ref{mainresult}, will provide a proof of what comes next for
the case of global solutions.
Using ansatz~\eqref{Ansatz} for $a$ and omitting the remainder terms gives
$$a(T)+\dot a(T)(t-T) = x_{-\infty}+u_{-\infty}T-C\ln|T| +
\left(u_{-\infty}-\frac{C}{T}\right)(t-T)\,.$$
Substituting this expression together with ansatz~\eqref{Ansatz} into
eq.~\eqref{Ansatz2}, again neglecting the remainders, provide the formula
$$C\left(\ln|T|-\ln|t|+\frac{t-T}{T}\right)=\eta_1\left(\ln|T|-\ln|t|+\frac{t-T}{T}\right)\,$$
which indicates $C=\eta_1$. Hence, we are led to the guess
\begin{equation*}
 a(t)=x_{-\infty}+u_{-\infty}t-\eta_1\ln|t|+o(\ln(|t|))
\end{equation*}
for $t\to-\infty$. Likewise, one obtains
\begin{equation*}
 b(t)=y_{-\infty}+v_{-\infty}t+\eta_2\ln|t|+o(\ln(|t|))
\end{equation*}
with
\begin{equation}\label{eta2}
 \eta_2=\kappa_b(1-v_{-\infty}^2)^{\frac32}\left[\frac{1-u_{-\infty}^2}{(u_{-\infty}-v_{-\infty})^2}\right]\,.
\end{equation}

\section{Main result and its proof}
\label{central}

Based on the information about the asymptotic behavior of solutions that was
provided in Section~\ref{asymptotes} we can now make our main result precise. 
As Banach space for the potential solutions to the FST equations~\eqref{WF}
we employ the space $\mathcal B$ of pairs of trajectories
$(a,b)\in\mathcal C^1(\mathbb R,\mathbb R^2)$ equipped with the norm
\begin{equation}\label{Norm}
    \|(a,b)\|:=\max\left(|a(0)|,|b(0)|,\|\dot a\|_\infty,\|\dot
    b\|_\infty\right)\,. 
\end{equation}
In this notation our main result reads:
\begin{theorem}[Global Existence]
    \label{mainresult}
    Let $$x_{-\infty},\,y_{-\infty}\in\mathbb R,\,\,-1<u_{-\infty}<v_{-\infty}<1$$ and
    \begin{equation}\label{Asymptoten}
        \begin{split}
            x(t)&=x_{-\infty}+u_{-\infty}t-\eta_1\ln|t|\,\\
            y(t)&=y_{-\infty}+v_{-\infty}t+\eta_2\ln|t| \qquad \text{for }t<-1
        \end{split}
    \end{equation}
    with $\eta_1$ and $\eta_2$ defined in eq.~\eqref{eta1}, \eqref{eta2}. Then,
    the following statements hold true:
    \begin{enumerate}[(i)]
        \item There exists a global solution $(a,b)\in C^\infty(\mathbb
            R,\mathbb R^2)$ to the FST equations~\eqref{WF} with
            $$a(t)=x(t)+O\left(\frac{\ln|t|}t\right),\,b(t)=y(t)+O\left(\frac{\ln|t|}t\right)
            \qquad\text{for }t\to-\infty.$$         \item Let furthermore $T_0<-1$ be a sufficiently large negative number such that
            $x(t)>y(t)$ and $|\dot x(t)|,|\dot y(t^\prime)|<1$ for all $t\le
            T_0$ and $t^\prime\le T_0^+$ is fulfilled, and let $(a_T,b_T)_{T\leq T_0}$ denote
            the family of conditional solutions to the FST equations~\eqref{WF}
            inferred by Theorem~\ref{conditional} satisfying the initial
            conditions \eqref{Anfangsbed}, \eqref{T+-}.  Then, there is a
            sequence $(T_n)_{n\in\mathbb N}$ with
            $\lim_{n\to\infty}T_n=-\infty$ fulfilling
            $$\lim_{n\to\infty} \| (a_{T_n},b_{T_n})-(a,b)\|=0.$$
    \end{enumerate}
\end{theorem}

The proof will be given at the end of this section. Before we will collect some
more technical results.  At the core of our proof are estimates that show a
uniform closeness of the conditional solutions $(a_T,b_T)$ and their velocities
$(\dot a_T,\dot b_T)$ to the asymptotes $(x,y)$ and the asymptotic velocities
$(u_{-\infty},v_{-\infty})$, respectively.
\begin{lemma}\label{lem:1}
Let $T_0$ be a sufficiently large negative number as in Theorem~\ref{mainresult}.
 There are $t_0<T_0$ and $C>0$ such that, for all $T\le t_0$ and $t\in[T,t_0]$,
 \begin{equation}\label{Geschw}
  \begin{split}
   &u_{-\infty} < \dot a_T(t) \le u_{-\infty} -\frac{C}{t}\,, \\
   &v_{-\infty} > \dot b_T(t) \ge v_{-\infty} +\frac{C}{t}
  \end{split}
 \end{equation}
and furthermore, for all $t\le t_0$,
 \begin{equation}\label{Orte}
 \left|a_T(t)-x(t)\right|, \left|b_T(t)-y(t)\right|<C\frac{\ln|t|}{|t|}
 \end{equation}
 with asymptotes $(x,y)$ defined in eq.~\eqref{Asymptoten}, \eqref{eta1}, \eqref{eta2}.
\end{lemma}
Lemma~\ref{lem:1} will be proven with the help of two further lemmata.

Note that the pair of asymptotes $(x,y)$ scatter apart for $t\to-\infty$. The
time instant $t_0$, thus, has to be chosen to be a sufficiently large negative
number such that the choice \eqref{Anfangsbed}, \eqref{T+-} of the initial conditions
on $x$ and $y$ ensures a sufficient decay of the Coulomb terms
$\frac1{\left(a_T(t)-b_T(t_2^\pm)\right)^2}$ in the FST~\eqref{WF} equations for $t\le
t_0$.  At first the estimates \eqref{Geschw} and \eqref{Orte} can
be more easily derived from the FST~\eqref{WF} equations under additional assumptions on
the distance and relative velocity of $a$ and $b$ that ensure such a decay:
\begin{lemma}\label{lem:2}
Let $T_0$ be a sufficiently large negative number as in Theorem~\ref{mainresult}.
Furthermore, let also $t_0<T_0$ be sufficiently large, $T\le t_0$ and
$$\mu:=\frac{u_{-\infty}-v_{-\infty}}2<0\,.$$ Assume that there is
a $t^*\in]T,t_0]$
such that
 \begin{equation}\label{Streuung}
 a_T(t)-b_T(t)\ge\mu t,\quad \dot a_T(t)-\dot b_T(t)\le\mu\quad\text{ for all }t\in[T,t^*]\,.
 \end{equation}
 Then, estimates \eqref{Geschw} and \eqref{Orte} hold true for all $t\in[T,t^*]$ and a suitable $C>0$.
\end{lemma}
Nevertheless, the computations are tedious and therefore deferred to
section~\ref{lemma2}.\\

At first sight, this might look circular: We can prove \eqref{Geschw},
\eqref{Orte} under the decay assumption \eqref{Streuung}, but in order to prove
the latter, something like \eqref{Geschw}, \eqref{Orte} (at least in a slightly
weakened form) seems to be necessary. However, by definition of $x$ and $y$,
$\lim_{t\to-\infty}\left(\dot x(t)-\dot y(t)\right)=2\mu$, so that
the initial
conditions \eqref{Anfangsbed} fulfill
\begin{equation}\label{Streuung2}
a_T(T)-b_T(T)=x(T)-y(T)>\mu T
\end{equation}
and $$\dot a_T(T)-\dot b_T(T)=\dot x(T)-\dot y(T)<\mu$$ 
sufficiently large negative $T$. Therefore, if $t_0$ is adapted
conveniently, continuity of $(a_T$,$b_T)$
implies that at least a small interval $[T,t^*]$, potentially 
with $t^*<t_0$, on which
\eqref{Streuung} and thus, by Lemma~\ref{lem:2}, \eqref{Geschw} and \eqref{Orte}, hold
true, exists. But if \eqref{Geschw} and \eqref{Orte} hold
on some interval, one can readily show that \eqref{Streuung} is satisfied on a
larger one:
\begin{lemma}\label{lem:3}
 Let $t_0$ be a sufficiently large negative number, $T$ as in Lemma~\ref{lem:2}, and
 $t^*\in]T,t_0]$ such that estimates \eqref{Geschw} and \eqref{Orte} hold true
 for all $t\in[T,t^*]$. Then, there exists $t^{**}\in]t^*,t_0]$ such that
 \begin{equation}\label{cont-claim}
 a_T(t)-b_T(t)\ge\mu t,\quad \dot a_T(t)-\dot b_T(t)\le\mu\quad\text{ for all }t\in[T,t^{**}]\,.
 \end{equation}
\end{lemma}
By iterating Lemma~\ref{lem:2} and \ref{lem:3}, one then arrives at Lemma~\ref{lem:1}.
\begin{proof}[Proof of Lemma~\ref{lem:3}]
By assumption~\eqref{Orte} and definition~\eqref{Asymptoten} of $x$ and $y$,
 \begin{equation}\label{dist}
  \begin{split}
  a_T(t)-b_T(t)&\ge x(t)-y(t)-\frac{2C}{\sqrt{|t|}}\\
  &=\mu t+\left(x_{-\infty}-y_{-\infty}+|\mu||t|-(\eta_1+\eta_2)\ln|t|)-\frac{2C}{\sqrt{|t|}}\right)
  \end{split}
 \end{equation}
for $t\in[T,t^*]$. The bracket becomes positive for sufficiently large $|t_0|$ and $t\le t_0$, so
$$a_T(t)-b_T(t)>\mu t$$
in this case and, by continuity, ``$\ge$'' is satisfied on on
some larger closed interval. Moreover, by assumption~\eqref{Geschw},
\begin{align}
    \label{vel}
 \dot a_T(t)-\dot b_T(t) \le u_{-\infty}-v_{-\infty}+\frac{C}{|t|}
 \le \mu+\left(\mu+\frac{C}{|t_0|}\right)
\end{align}
for $t\in[T,t^*]$. Here, the bracket becomes negative for sufficiently large $|t_0|$, implying
$$\dot a_T(t)-\dot b_T(t)<\mu\text{ for }t\in[T,t^*]$$ and, again,  
``$\le$'' holds on some larger interval. This implies the existence of a
$t^{**}\in ]t^*,t_0]$ such that \eqref{cont-claim} holds.
\end{proof}
Now we prove Lemma~\ref{lem:1}: 
\begin{proof}[Proof of Lemma~\ref{lem:1}]
 Fix a sufficiently large negative $t_0$ such that both Lemma~\ref{lem:2} and \ref{lem:3} are applicable
 and inequality~\eqref{Streuung2} holds true for all $T\le t_0$.
 Iterated application of these two lemmata reveals that inequalities~\eqref{Streuung} hold true on $[T,t_0]$: Assuming the contrary, fix $T$ as required and define
 $$\tilde t:=\inf\left\{t\in[T,t_0]\mid a_T(t)-b_T(t)<\mu t \vee \dot
 a_T(t)-\dot b_T(t)>\mu\right\}\,.$$ 
 By continuity of $a_T$ and $b_T$, inequalities~\eqref{Streuung} are fulfilled on $[T,\tilde t]$, so Lemma~\ref{lem:2} implies that Lemma~\ref{lem:3} is applicable with $t^*=\tilde t$, resulting in a contradiction to the definition of $\tilde t$.
 
 By definition \eqref{ext}, $a_T(t)-x(t)=0=b_T(t)-y(t)$ for $t\le T$, so
 inequalities~\eqref{Orte} are also valid for such times $t$.
\end{proof}

Beyond the estimates in Lemma~\ref{lem:1} for large negative $t$, we need a priori
estimates for all $t\in\mathbb R$ that prevent the terms in the FST equations
\eqref{WF}
from becoming singular:
\begin{lemma}\label{lem:4}
 Let $t_0$ be a sufficiently large negative number as in Lemma~\ref{lem:1}.
 There are $V\in[0,1[$ and $D>0$ such that all conditional solutions for $(a_T,b_T)$
     with $T\le t_0$ corresponding to initial conditions \eqref{Anfangsbed}
     with
     asymptotes in \eqref{Asymptoten} satisfy
 \begin{equation}\label{V}
 \|\dot a_T\|_\infty,\|\dot b_T\|_\infty\le V
 \end{equation}
 and
 \begin{equation}\label{D}
 a_T(t)-b_T(t)\ge D(1+|t|)\text{ for all }t\in\mathbb R\,.
 \end{equation}
\end{lemma}
Before we discuss the proof, we remark that \eqref{V} is a corollary to
Proposition 2.1 from \cite{DirkGuenter}, but \eqref{D} is deduced from
Lemma~\ref{lem:1}. This poses no problem since \eqref{D} will only be exploited
in the proof of the main result Theorem \ref{mainresult}, whereas the estimate
\eqref{V} will be used throughout the proof of Lemma~\ref{lem:2}.  In the proof
of Lemma~\ref{lem:4}, and
whenever dealing with the advanced and retarded times, the following estimates,
already mentioned before in \eqref{eq:t+-est} and proven in \cite[Lemma
2.1]{DirkGuenter}, will be useful:
\begin{lemma}\label{lem:t+-}
 For any $C^1$-trajectories $(a,b)$ with $\|\dot a\|_\infty,\|\dot b\|_\infty\le C<1$ and $a(t)>b(t)$ for all $t\in\mathbb R$,
 the advanced and retarded times $t_i^\pm$ introduced in \eqref{t+-} are globally well-defined and
 \begin{equation*}
  \begin{split}
 &\frac{a(t)-b(t)}2\le a(t)-b(t_2^\pm(a(t),b,t))\le\frac{a(t)-b(t)}{1-\|\dot b\|_\infty}\,, \\
 &\frac{a(t)-b(t)}2\le a(t_1^\pm(a,b(t),t))-b(t)\le\frac{a(t)-b(t)}{1-\|\dot a\|_\infty} \,.
 \end{split}
 \end{equation*}
\end{lemma}
 
\begin{proof}[Proof of Lemma~\ref{lem:4}:]
 According to the proof of Proposition 2.1 in \cite{DirkGuenter} (the equation after 29), for any conditional solution with initial data $a_T(T)=a_0$, $\dot a_T(T)=\dot a_0$, $\left.b_T\right|_{[T^-,T^+]}=b_0$ the estimate
 \begin{equation*}
  \sup_{t\ge T}|\dot a_T(t)|\le\sqrt{1-\frac{(1-\|\dot b_0\|_\infty)^2}{\left(\frac4{\sqrt{1-\left(\max(\dot a_0,\|\dot x_0\|_\infty)\right)^2}+\frac{3\kappa_a}{a_0-b_0(T)}}\right)^2}}
 \end{equation*}
 holds true. In \cite{DirkGuenter}, it was formulated only for the case $T=0$, but it holds without change for different $t$. $x_0$ denotes a reference trajectory, which, in the case at hand, can be chosen in such a way that $\|\dot x_0\|_\infty=u$ where $u$ is an arbitrary number between $\sup_{t\le T_0}|\dot x(t)|$ and 1.
 Substituting as initial data the corresponding segments of the asymptotes $x,y$ and recalling from their explicit form~\eqref{Asymptoten} and the choice of $t_0\le T_0$ that their velocities are bounded away from 1 and the distances $a_0-b_0(T)$ from 0, we conclude
 \begin{equation*}
  \sup_{T\le t_0}\sup_{t\ge T}|\dot a_T(t)|<1\,.
 \end{equation*}
Since $a_T(t)=x(t)$ for $t\le T$, we get $$\sup_{T\le t_0}\|\dot a_T\|_\infty<1$$
and, by the analogous reasoning for $b$, the existence of $V\in]0,1[$ such that $$\|\dot a_T\|_\infty,\|\dot b_T\|\le V$$ for all $T\le t_0$. Estimate (30) in \cite{DirkGuenter} then implies the existence of a uniform lower bound for $\inf_{t\ge T}\left(a_T(t)-b_T(t)\right)$ and $T\le t_0$ which, as the distance of the asymptotes is bounded away from 0, can again be extended to a bound $\tilde D$ valid for all $t\in\mathbb R$.

In order to prove \eqref{D}, it remains to find a negative upper bound for $\dot a_T(t)-\dot b_T(t)$ for sufficiently large negative $t$ and a positive lower bound for sufficiently large $t$, both uniform in $T$. The first of them is immediately given by Lemma~\ref{lem:1}, inequality~\eqref{Geschw}, and the fact that $u_{-\infty}-v_{-\infty}<0$. To find the latter, we employ a proof by contradiction. We observe that, according to Lemma~\ref{lem:1}, we can find $S$, $V^\prime$ and $D^\prime$ such that $a_T(S)-b_T(S)\le D^\prime$ and $\dot a_T(S)-\dot b_T(S)\ge V^\prime$ hold true for all $T\le t_0$.
Assuming now that $\dot a_T(t)-\dot b_T(t)\le0$ for all $t\in[S,\tilde S]$ with $\tilde S>S$, Lemma~\ref{lem:t+-} together with the uniform velocity bound~\eqref{V} implies
$$a(t)-b(t_2^+)\le\frac{D^\prime}{V}$$ for these $t$
and thus, due to~\eqref{Beschleunigung} and the equation~\eqref{WF} of motion, yields a uniform lower bound
\begin{equation*}
 \ddot a_T(t)\ge\frac{\kappa_a}2(1-\dot a_T(t)^2)^{\frac32}\frac{1-\dot b_T(t_2^+)}{1+\dot b_T(t_2^+)}\frac1{\left(a_T(t)-b_T(t_2^+)\right)^2}
 \ge \frac{\kappa_a(1-V^2)^{\frac32}(1-V)}{2(1+V){D^\prime}^2}=:C_1\,.
\end{equation*}
on the acceleration of particle $a$. Analogously, one obtains
$$\ddot b_T(t)\le-C_2$$ and, contrary to the assumption,
$$\dot a_T(\tilde S)-\dot b_T(\tilde S)\ge V^\prime+\int_S^{\tilde S}(C_1+C_2)\text ds$$
must be positive for a sufficiently large $\tilde S$ independent of $T$.

From time $\tilde S$ on, one gets a uniform upper bound on $a(t)-b(t_2^+)$ by assuming that $a$ moved to the right and $b$ to the left with the maximal possible velocity $V$ already from time $S$ on. The resulting lower bound on $\ddot a_T(t)$ for $t\ge\tilde S$ and the analogous upper one on $\ddot b_T(t)$ lead to the desired uniform positive lower bound on $\dot a_T(t)-\dot b_T(t)$ e.g. for times $t\ge\tilde S+1$.
\end{proof}

The sequence $\left(a_{T_n},b_{T_n}\right)$ as in the Theorem can now be found
by a compactness argument. We formulate an appropriate generalization of the
Arzela-Ascoli theorem, the proof of which can be found in \cite[Lemma 2.3]{DirkGuenter}:
\begin{lemma}\label{lem:Arzela}
 If a sequence $f_1,f_2,\dots$ of bounded continuous functions on $\mathbb R$ is uniformly bounded and equicontinuous and
 \begin{equation}\label{Kompaktheitsbed}
 \lim_{S\to\infty}\sup_{t>S,n\in\mathbb
 N}\max\left\{|f_n(t)-f_n(S)|,|f_n(-t)-f_n(-S)|\right\}=0\,,
 \end{equation}
 then it has a uniformly convergent subsequence.
\end{lemma}
Finally, with these technical lemmata, we can prove our main result:
\begin{proof}[Proof of Theorem \ref{mainresult}]
    Recall the family of solutions on the half-line $(a_T,b_T)_{T\leq T_0}$
    corresponding to the initial conditions \eqref{Anfangsbed} that are provided
    by Theorem~\ref{conditional}. The proof is divided in several steps:\\

    \emph{1) Existence of an accumulation point:}
 We start by observing that $(\dot a_T)_{T\le t_0}$ is uniformly bounded
 according to Lemma~\ref{lem:4}. For $t\ge T$, according to the FST
 equations~\eqref{WF} and Lemmata~\ref{lem:t+-} and \ref{lem:4},
 \begin{align*}
  0<\ddot a_T(t)\le\frac{\kappa_a}2\left(1-V^2\right)^{\frac32}\left[2\cdot\frac2{1-V}\frac{4}{D^2(1+|t|)^2}\right]
  \le\frac{8\kappa_a}{(1-V)D^2(1+|t|)^2}\,.
 \end{align*}
For $t\le T$, $a_T(t)=x(t)$, so, by definition~\eqref{Asymptoten},
\begin{equation*}
 0<\ddot a_T(t)=\frac{\eta_1}{t^2}
\end{equation*}
and therefore, $$0<\ddot a_T(t)\le\frac{C}{1+t^2}$$ holds for a suitable $C>0$
for all $t\leq T\leq t_0$. Consequently, for all $s,t\in\mathbb R$ with $s<t$, we have $$|\dot
a_T(t)-\dot a_T(s)|\le\int_s^t\frac{C}{1+u^2}\text du\,,$$ implying that 
$(\dot a_T)_{T\le t_0}$ is equicontinuous and that condition \eqref{Kompaktheitsbed} from
Lemma~\ref{lem:Arzela} is fulfilled, 
so that a uniformly convergent sequence
$\left(\dot a_{T_n}\right)_{n\in\mathbb N}$ with
$T_n\xrightarrow{n\to\infty}-\infty$ exists. 

Furthermore, choosing $t\le t_0$ and using
estimate~\eqref{Orte} from Lemma~\ref{lem:1} together with $|\dot a_T(t)|\le1$,
gives
\begin{equation*}
 \left|a_T(0)\right|=\left|a_T(t)+\int_t^0\dot a_T(s)\text ds\right|\le|x(t)|+|t|\,.
\end{equation*}
This shows us that $(a_T(0))_{T\le t_0}$ is bounded. Therefore, for a suitable
subsequence $(T_{n_k})_{k\in\mathbb N}$ of $(T_n)_{n\in\mathbb N}$, also
$(a_{T_{n_k}}(0))_{k\in\mathbb N}$ converges. From
$(T_{n_k})_{k\in\mathbb N}$, we can, by an analogous reasoning, extract a subsequence
$(T_{n_{k_i}})_{i\in\mathbb N}$ such that also $(\dot
b_{T_{n_{k_i}}})_{i\in\mathbb N}$
and $(b_{T_{n_{k_i}}}(0))_{i\in\mathbb N}$ converge. In conclusion, there is a
subsequence, with slight abuse of notation again denoted by
$(a_{n},b_{n})_{n\in\mathbb N}$, which  converges
with respect to the norm~\eqref{Norm}.\\

\emph{2) Estimates for the accumulation point.} Estimate~\eqref{Orte} in Lemma~\ref{lem:1}, as well as Lemma~\ref{lem:4}, hold
also true for $$(a,b):=\lim_{n\to\infty}(a_n,b_n)\,.$$ This can be seen by
applying these estimates to $(a_n(t),b_n(t))$ for fixed $t$ and exploiting the
convergence $a_n(t)\to a(t)$ and $b_n(t)\to b(t)$ for
$n\to\infty$. Therefore, 
\begin{align}
    \label{order}
a(t)=x(t)+O\left(\frac{\ln|t|}{|t|}\right)\,,\qquad
b(t)=y(t)+O\left(\frac{\ln|t|}{|t|}\right)\,. 
\end{align}

\emph{3) Fulfillment of the FST equations.}  First, we observe that, according to Lemma~\ref{lem:4} and
\ref{lem:t+-}, for $(a,b)$, the advanced and retarded times $t_i^\pm$ in
\eqref{t+-} are well-defined, and so is the right-hand sides of the FST
equations. It remains to show that $(a,b)$
solve the FST equations. The integral equation fulfilled by $a_n$ is given by
\begin{align}
    \label{integral}
    \begin{split}
 \dot a_n(t)=
 \dot a_n(T)+\frac{\kappa_a}{2}\int_T^t(1-\dot a_n(u)^2)^{\frac32} &\left[\frac{1+\dot b_n\left(t_2^-\right)}{1-\dot b_n\left(t_2^-\right)}\frac1{\left(a_n(u)-b_n(t_2^-)\right)^2}\right. \\
 &\,\left.+ \frac{1-\dot b_n\left(t_2^+\right)}{1+\dot b_n\left(t_2^+\right)}\frac1{\left(a_n(u)-b_n(t_2^+)\right)^2}\right]\text du \,.
 \end{split}
\end{align}
Hence, it suffices to show that we may exchange the limit $n\to\infty$ with the
integration. For this it is sufficient to show that the integrand converges uniformly on compact
intervals $[-S,S]$; and likewise one has to repeat the proof for $\dot
b_n$.  However, we know that $a_n,\,\dot a_n,\,b_n,\,\dot b_n$ converge uniformly  thanks
to the definition of the norm, the
denominators are bounded away from zero by Lemma~\ref{lem:4}, and furthermore,
the uniform
convergence of $t_2^\pm(a_n(\cdot),b_n(\cdot),\cdot)=:t_{2,n}^\pm(\cdot)$ to
$t_2^\pm(a(\cdot),b(\cdot),\cdot)=:t_2^\pm(\cdot)$ follows from the estimate
\begin{align*}
 \left|t_{2,n}^\pm(t)-t_2^\pm(t)\right|
 =&\left|t\pm a_n(t)\mp b_n(t_{2,n}^\pm(t))-t\mp a(t)\pm b(t_2^\pm(t))\right|\\
 \le&\left|a_n(t)-a(t)\right| + \left|b_n(t_{2,n}^\pm(t))-b_n(t_2^\pm(t))\right| + \left|b_n(t_2^\pm(t))-b(t_2^\pm(t))\right|\\
 \le& \left|a_n(0)-a(0)\right|+|S|\|\dot a_n-\dot a\|_\infty+V\left|t_{2,n}^\pm(t)-t_2^\pm(t)\right|\\
 &+\left|b_n(0)-b(0)\right|+|t_2^\pm(t)|\|\dot b_n-\dot b\|_\infty \\
 \le& \frac{1+|S|+\sup_{t\in[-S,S]}|t_2^\pm(t)|}{1-V}\left\|(a_n-a,b_n-b)\right\|\,.
\end{align*}
Note that the supremum is finite since $t_2^\pm$ is continuous. Hence, we may
interchange the limit $n\to\infty$ with the integral in \eqref{integral}; and
likewise for the corresponding integral equation for $b_n$. By Theorem \ref{conditional} and the uniform
convergence, we know that $(a,b)$ is smooth. Hence, we may take the derivate of
the integral equations \eqref{integral} recover the FST equations \eqref{WF}.\\

In summary, the last step proves the existence of a smooth global solution to
the FST equations \eqref{WF} that results from the convergence of the
sequence $(a_n,b_n)_{n\in\mathbb N}$, which furthermore obeys the asymptotic
behavior \eqref{order}, which concludes the proof.
\end{proof}

\subsection{Proof of Lemma~\ref{lem:2}}\label{lemma2}
In this last section we provide the remaining proof of Lemma~\ref{lem:2}.
Recall that Lemma \ref{lem:2} is supposed to ensure the claims of Lemma
\ref{lem:1}, i.e., estimate \eqref{Geschw} and \eqref{Orte}, under the stronger
condition \eqref{Streuung}. In the following we prove both claims separately
denoted by Part I and Part II.
We only show the estimates for $a_T$, the ones for $b_T$ are obtained analogously.
In our notation, $C$ will denote finite and positive constants that may vary
from line to line.

Ideally we would like to proof Lemma~\ref{lem:2} assuming only that $T\le T_0$
and that $t^*\in]T,T_0]$ exists such that estimates~\eqref{Streuung} hold true
for all $t\in[T,t^*]$, and consider $t\in[T,t^*]$.  However, several steps in
the proof of Lemma~\ref{lem:2}, including the auxiliary lemmata in this section,
will only hold under finitely many additional conditions of the form that $t$ is
a sufficiently large negative number, i.e., $t\le t_n$ for finitely many $t_n$.
Since we pick up these extra constraints along the way in the proof of
Lemma~\ref{lem:2}, we possibly have to adjust $T_0$ each time and start over the
with the proof -- at most finitely many times. This is unproblematic since all
previous estimates hold also for larger negative values of
$T_0$.

Therefore, in order to keep the presentation reasonably short we employ a slight
abuse of notation to avoid repetition of the proof: Instead of keeping $T_0$
fixed, we adjust its value from $T_0$ to $T_0\wedge t_n$ each time we pick up
another constraint $t\leq t_n$, keeping in mind that in the end, the proof
will only hold for $t_0:=\min_nt_n$ -- exactly in the form
given in Lemma \ref{lem:2}.

\begin{proof}[Proof of Lemma~\ref{lem:2}, Part I:  Estimate~\eqref{Geschw}]
    We observe that $\ddot a_T(t)>0$ for $t>T$ and, by
    definition~\eqref{Asymptoten}, also
    $\ddot a_T(t)=\frac{\eta_1}{t^2}>0$ holds true for $t<T$. Hence,
    $\lim_{t\to-\infty}\dot a_T(t)=u_{-\infty}$ implies $\dot
    a_T(t)>u_{-\infty}$ for all $t$. Using the integrated equation of
    motion~\eqref{WFintTt} (for $a_T$ instead of $a$), the velocity
    estimate~\eqref{V}, Lemma~\ref{lem:t+-} and assumption~\eqref{Streuung}, we
    find
\begin{align*}
 \dot a_T(t)\le&\dot a_T(T)+C\int_T^t\left[\frac1{\left(a_T(s)-b_T(t_2^-)\right)^2}+\frac1{\left(a_T(s)-b_T(t_2^+)\right)^2}\right]\text ds\\
 \le& \dot a_T(T)+C\int_T^t\frac1{\left(a_T(s)-b_T(s)\right)^2}\text ds \\
 \le&\dot a_T(T)+C\int_T^t\frac1{s^2}\text ds\,.
\end{align*}
Since $\dot a_T(T)=\dot x(T)=u_{-\infty}-\frac{\eta_1}T$, estimate~\eqref{Geschw} follows.
\end{proof}

In order to prove estimate~\eqref{Orte} and with it provide Part II of the
proof of Lemma~\ref{lem:2}, we use the fact that equation~\eqref{WFint} also
holds true for $(a_T,b_T)$ instead of $(a,b)$ as long as $t\in[T,
t^*]$; in the following we refer to \eqref{WFint} in the sense of
$(a,b)$ replaced by $(a_T,b_T)$. Our
goal is to employ this formula in order to estimate the distance $|a_T(t)-x(t)|$
by observing cancellations or asymptotically vanishing terms.
Term~\eqref{WFinta}
now reads
\begin{equation}\label{avonT}
\begin{split}
a_T(T)+\dot a_T(T)(t-T)
=&x_{-\infty}+u_{-\infty}T-\eta_1\ln|T|+\left(u_\infty-\frac{\eta_1}T\right)(t-T)\\
=&x_{-\infty}+u_{-\infty}t-\eta_1\ln|T|-\eta_1\frac{t-T}T\,.
\end{split}
\end{equation}
In order to gain some intuition about the terms, we observe that 
the first two summands cancel with the ones in
definition~\eqref{Asymptoten} of $x$. Moreover, as indicated in the
section~\ref{asymptotes}, the factor
of term~\eqref{WFintb} in front of $(t-T)$, multiplied by $T$, converges to
$\eta_1$. Thus,
term~\eqref{WFintb} should cancel the last summand in
\eqref{avonT} asymptotically . Term~\eqref{WFintc} approaches
$$-\eta_1\ln|t|-\frac{\eta_1}2\ln\left(-\frac{(u_{-\infty}-v_{-\infty})^2}{1-v_{-\infty}^2}\right)$$
and, likewise, term~\eqref{WFintd},
$$\eta_1\ln|T|+\frac{\eta_1}2\ln\left(-\frac{(u_{-\infty}-v_{-\infty})^2}{1-v_{-\infty}^2}\right)\,.$$
Finally, \eqref{WFinte} and \eqref{WFintf} are expected to vanish separately.
Correspondingly, it is convenient to group the terms 
\begin{equation}\label{ASumme}
 |a_T(t)-x(t)|\le\sum_{n=1}^5|A_n(t)|\,,
\end{equation}
where
 \begin{align}\label{A1}
\begin{split}
 &A_1(t) := \frac{\kappa_a}2\left(1-\dot a_T(T)^2\right)^{\frac32}\left[\frac{1+\dot b_T(t_2^-(T))}{\dot a_T(T)-\dot b_T(t_2^-(T))}\frac T{a_T(T)-b_T(t_2^-(T))}\right. \\
 &\hphantom{A_1(t):= \frac{\kappa_a}2\left(1-\dot a_T(T)^2\right)^{\frac32}\left[\right.}  + \left.\frac{1-\dot b_T(t_2^+(T))}{\dot a_T(T)-\dot b_T(t_2^+(T))}\frac T{a_T(T)-b_T(t_2^+(T))}\right] \frac{t-T}T
 -\eta_1\frac{t-T}T \,,
 \end{split}
 \end{align}
 \begin{align}\label{A2}
 \begin{split}
 &A_2(t):= \frac{\kappa_a}2\left(1-\dot a_T(t)^2\right)^{\frac32}\left[\frac{1-\dot b_T(t_2^-(t))^2}{\left(\dot a_T(t)-\dot b_T(t_2^-(t))\right)^2}\ln\left(a_T(t)-b_T(t_2^-(t))\right)\right. \\
 &\hphantom{A_2(t):=\frac{\kappa_a}2\left(1-\dot a_T(t)^2\right)^{\frac32}\left[\right.}
 + \left. \frac{1-\dot b_T(t_2^+(t))^2}{\left(\dot a_T(t)-\dot b_T(t_2^+(t))\right)^2}\ln\left(a_T(t)-b_T(t_2^+(t))\right)\right]  \\
 &\hphantom{A_2(t):=}-\eta_1\ln|t|-\frac{\eta_1}2\ln\left(-\frac{(u_{-\infty}-v_{-\infty})^2}{1-v_{-\infty}^2}\right) \,,
 \end{split}
 \end{align}
 \begin{align}\label{A3}
 A_3(t):=A_2(T) \,,
 \end{align}
 \begin{align}\label{A4}
 \begin{split}
 &A_4(t):= \frac{\kappa_a}2\int_T^t \left[\frac{\text d}{\text ds}\left(\left(1-\dot a_T(s)^2\right)^{\frac32}\frac{1-\dot b_T(t_2^-(s))^2}{\left(\dot a_T(s)-\dot b_T(t_2^-(s))\right)^2}\right)\right]\ln\left(a_T(s)-b_T(t_2^-(s))\right) \\
  &\hphantom{\frac{\kappa_a}2\int_T^t} + \left[\frac{\text d}{\text ds}\left(\left(1-\dot a_T(s)^2\right)^{\frac32}\frac{1-\dot b_T(t_2^+(s))^2}{\left(\dot a_T(s)-\dot b_T(t_2^+(s))\right)}\right)\right]\ln\left(a_T(s)-b_T(t_2^+(s))\right)\text ds \,,
  \end{split}
  \end{align}
  \begin{align}\label{A5}
  \begin{split}
 &A_5(t):= \frac{\kappa_a}2\int_T^t \int_T^s \left[\frac{\text d}{\text dr}\left(\left(1-\dot a_T(r)^2\right)^{\frac32}\frac{1+\dot b_T(t_2^-(r))}{\dot a_T(r)-\dot b_T(t_2^-(r))}\right)\right]\frac1{a_T(r)-b_T(t_2^-(r))} \\
  &\hphantom{\frac{\kappa_a}2\int_T^t\int_T^s} + \left[\frac{\text d}{\text dr}\left(\left(1-\dot a_T(r)^2\right)^{\frac32}\frac{1-\dot b_T(t_2^+(r))}{\dot a_T(r)-\dot b_T(t_2^+(r))}\right)\right]\frac1{a_T(r)-b_T(t_2^+(r))} \text dr\text ds\,,
  \end{split}
 \end{align}
and show, term by term, that 
\begin{align}
    \label{eq:Aest}
   |A_n(t)|\le\frac C{\sqrt{|t|}} 
\end{align} holds. For later
use, we remark that a different rearrangement (with $A_2$ on the other side as
in estimate~\eqref{ASumme}) gives
\begin{equation}\label{ASumme2}
 \begin{split}
  &\left|a_T(t)-x_{-\infty}-u_{-\infty}t-\frac{\eta_1}2\ln\left(\frac{(u_{-\infty}-v_{-\infty})^2}{1-v_{-\infty}^2}\right)\right. \\
  &+ \frac{\kappa_a}2\left(1-\dot a_T(t)^2\right)^{\frac32}\left[\frac{1-\dot b_T(t_2^-(t))^2}{\left(\dot a_T(t)-\dot b_T(t_2^-(t))\right)^2}\ln\left(a_T(t)-b_T(t_2^-(t))\right)\right. \\
 &\hphantom{+ \frac{\kappa_a}2\left(1-\dot a_T(t)^2\right)^{\frac32}\left[\right.} \left.+ \left. \frac{1-\dot b_T(t_2^+(t))^2}{\left(\dot a_T(t)-\dot b_T(t_2^+(t))\right)^2}\ln\left(a_T(t)-b_T(t_2^+(t))\right)\right]\right|\le\sum_{n\neq2}\left|A_n(t)\right|\,.
 \end{split}
\end{equation}
The following three lemmata will provide the essential ingredients for the
estimate of \eqref{ASumme}.
\begin{lemma}\label{lem:apunkt}
    For sufficiently large negative numbers $T,t$ such that $T\leq t$, it holds
    that
 \begin{subequations}
  \begin{align}
   &\dot a_T(t_1^+)-\dot a_T(t)<\frac C{|t|}\,,\quad\dot b_T(t)-\dot b_T(t_2^+)<\frac{C}{|t|}\,,\label{apunkt1}\\
   &\dot a_T(t)-\dot b_T(t_2^-)\le\mu\,,\quad \dot a_T(t_1^-)-\dot b_T(t)\le\mu\,,\label{apunkt2}\\
   &\dot a_T(t_1^+)-\dot b_T(t)\le\frac\mu2\,,\quad \dot a_T(t)-\dot b_T(t_2^+)\le\frac\mu2\,.\label{apunkt3}
  \end{align}
 \end{subequations}
\end{lemma}

\begin{proof}
 In the integrated equation~\eqref{WF} of motion
 \begin{equation}\label{inttt+}
 \begin{split}
 &\dot a_T(t_1^+(t))-\dot a_T(t)
  = \frac{\kappa_a}{2}\int_t^{t_1^+(t)}(1-\dot a_T(s)^2)^{\frac32} \left[\frac{1+\dot b_T\left(t_2^-(s)\right)}{1-\dot b_T\left(t_2^-(s)\right)}\frac1{\left(a_T(s)-b_T(t_2^-(s))\right)^2}\right.\\
  &\hphantom{\dot a_T(t_1^+(t))-\dot a_T(t)  = \frac{\kappa_a}{2}\int_t^{t_1^+(t)}(1-\dot a_T(s)^2)^{\frac32} \left[\right.}
  \left.+ \frac{1-\dot b_T\left(t_2^+(s)\right)}{1+\dot b_T\left(t_2^+(s)\right)}\frac1{\left(a_T(s)-b_T(t_2^+(s))\right)^2}\right]\text ds\,,
  \end{split}
 \end{equation}
we shall estimate the denominators $a_T(s)-b_T(t_2^\pm(s))$ by terms of the form
$f(t)+g(t)(t-s)$ which will be possible even if the
integration can extend beyond the time up to which estimates~\eqref{Streuung}
are valid.  Due to the mean value theorem and the fact that $\ddot b_T<0$, which
implies $\dot b_T(t)<\dot b_T(s)$ for $t>s$, we find
\begin{equation*}
 b_T(t_2^-(s))\le b_T(t)-\dot b_T(t)(t-t_2^-(s))\,.
\end{equation*}
Definition~\eqref{t+-} implies $s-t_2^-(s)=a_T(s)-b_T(t_2^-(s))$, so that we
have
\begin{align*}
 a_T(s)-b_T(t_2^-(s))\ge& a_T(s)-b_T(t)+\dot b_T(t)(t-s)+\dot b_T(t)\left(a_T(s)-b_T(t_2^-(s))\right)\\
 \ge& \frac{a_T(s)-b_T(t)+\dot b_T(t)(t-s)}{1-\dot b_T(t)}\,.
\end{align*}
Employing the analogous estimate $a_T(s)\ge a_T(t_1^+(t))-\dot
a_T(t_1^+(t))\left(t_1^+(t)-s\right)$ we get
\begin{equation*}
 a_T(s)-b_T(t_2^-(s))\ge \frac{f(t)+g(t)(t-s)}{1-\dot b_T(t)}
\end{equation*}
for functions
$$f(t):= a_T(t_1^+(t))-b_T(t)-\dot a_T(t_1^+(t))\left(t_1^+(t)-t\right)
 =\left(1-\dot a_T(t_1^+(t))\right)\left(a_T(t_1^+(t))-b_T(t)\right)$$
 -- where, for the last equality,
 \begin{equation}\label{lemapunkt1}
  t_1^+(t)-t=a_T(t_1^+(t))-b_T(t)
 \end{equation}
 should be recalled --
 and
 $$g(t):= \dot b_T(t)-\dot a_T(t_1^+(t))\,.$$
Likewise, one obtains
$$a_T(s)-b_T(t_2^+(s))\ge \frac{f(t)+g(t)(t-s)}{1+\dot b(t)}\,.$$
Using both of these estimates in the integral equation \eqref{inttt+} together
with the bound on the velocities \eqref{V} in order to estimate $\dot b_T$,
performing the integration, and exploiting the identity~\eqref{lemapunkt1}, we find
\begin{align*}
 \dot a_T(t_1^+(t))-\dot a_T(t) \le& \int_t^{t_1^+(t)}\frac C{\left[f(t)+g(t)(t-s)\right]^2}\text ds\\
 =&\frac C{g(t)\left[f(t)+g(t)(t-t_1^+(t))\right]} - \frac C{g(t)\left[f(t)+g(t)(t-t)\right]}\\
 =&\frac{C\left(f(t)-\left[f(t)+g(t)(t-t_1^+(t))\right]\right)}{g(t)f(t)\left(\left[f(t)+g(t)(t-t_1^+(t))\right]\right)}\\
 =&\frac{C(t_1^+(t)-t)}{f(t)\left[f(t)-g(t)\left(t_1^+(t)-t\right)\right]}\\
 =&\frac{C}{\left(1-\dot a_T(t_1^+(t))\right)\left(1-\dot b_T(t)\right)\left(a_T(t_1^+(t))-b_T(t)\right)}\,.
\end{align*}
Here, we have resubstituted the functions $f$ and $g$ again.
Now Lemma~\ref{lem:t+-} and assumption \eqref{Streuung} allow us to conclude
\begin{equation*}
 \dot a_T(t_1^+(t))-\dot a_T(t) \le\frac{C}{a_T(t)-b_T(t)}\le\frac C{|t|}\,.
\end{equation*}
Since $\ddot b_T<0$, we also get
$$\dot a_T(t)-\dot b_T(t_2^-)\le\dot a_T(t)-\dot b_T(t)\le\mu$$
thanks to assumption \eqref{Streuung}. If $t$ is a sufficiently large negative number, recalling that $\mu<0$, \eqref{apunkt1} implies
$$\dot a_T(t_1^+)-\dot b_T(t)= \dot a_T(t_1^+)-\dot a_T(t)+\dot a_T(t)-\dot b_T(t)\le\frac\mu2\,.$$
We omit a proof of the remaining inequalities which are obtained by very similar
arguments.
\end{proof}

\begin{lemma}\label{lem:eta}
 For sufficiently large negative numbers $T,t$  such that $T\leq
 t$, it holds that
 \begin{equation*}
  \left|\kappa_a\left(1-\dot a_T(t)^2\right)^\frac32\frac{1-\dot b_T(t_2^\pm)^2}{\left(\dot a_T(t)-\dot b_T(t_2^\pm)\right)^2}-\eta_1\right|\le\frac C{|t|}\,.
 \end{equation*}
\end{lemma}
\begin{proof}
By definition~\eqref{eta1} of $\eta_1$, we find
\begin{align*}
 D(t):=&\left|\kappa_a\left(1-\dot a_T(t)^2\right)^\frac32\frac{1-\dot b_T(t_2^\pm)^2}{\left(\dot a_T(t)-\dot b_T(t_2^\pm)\right)^2}-\eta_1\right| \\
 \le& \kappa_a\left|\left(1-\dot a_T(t)^2\right)^\frac32\frac{1-\dot b_T(t_2^\pm)^2}{\left(\dot a_T(t)-\dot b_T(t_2^\pm)\right)^2}
 - \left(1-\dot a_T(t)^2\right)^\frac32\frac{1-v_{-\infty}^2}{\left(\dot a_T(t)-v_{-\infty}\right)^2}\right| \\
 &+\kappa_a\left|\left(1-\dot a_T(t)^2\right)^\frac32\frac{1-v_{-\infty}^2}{\left(\dot a_T(t)-v_{-\infty}\right)^2}
 - \left(1-u_{-\infty}^2\right)^\frac32\frac{1-v_{-\infty}^2}{\left(u_{-\infty}-v_{-\infty}\right)^2}\right|\,.
\end{align*}
The facts
\begin{equation}\label{lemeta2}
 \frac{\partial }{\partial u}\left[(1-u^2)^{\frac32}\frac{1-v^2}{(u-v)^2}\right] = \frac{(1-u^2)^{\frac12}(1-v^2)(-u^2+3uv-2)}{(u-v)^3}
\end{equation}
and
\begin{equation}\label{lemeta3}
 \frac{\partial }{\partial v}\left[(1-u^2)^{\frac32}\frac{1-v^2}{(u-v)^2}\right] = \frac{2(1-u^2)^{\frac32}(1-uv)}{(u-v)^3}\,,
\end{equation}
together with the mean value theorem and the bound $|\dot a_T|,|\dot b_T|\le1$ yield
\begin{equation}\label{lemeta1}
 D(t)
 \le\frac{C}{\left|\dot a_T(t)-v^*\right|^3}\left|v_{-\infty}-\dot b_T(t_2^\pm)\right| + \frac{C}{\left|u^*-v_{-\infty}\right|^3}\left|\dot a_T(t)-u_{-\infty}\right|
\end{equation}
for some $u^*\in[u_{-\infty}, \dot a_T(t)]$ and $v^*\in[\dot b_T(t_2^\pm), v_{-\infty}]$. Since $\ddot a_T>0, \ddot b_T<0$, we get
$$u_{-\infty}-v_{-\infty}\le u^*-v_{-\infty}\le\dot a_T(t)-v^*\le\dot a_T(t)-\dot b_T(t_2^\pm)\le\mu<0$$
thanks to Lemma~\ref{lem:apunkt}. Therefore, $|\dot a_T(t)-v^*|,
|u^*-v_{-\infty}|\ge C$ holds. Employing this bound together with \eqref{Geschw}
in
\eqref{lemeta1} concludes the proof.
\end{proof}

\begin{lemma}\label{lem:letztes}
 For all $t\in\mathbb R$ it holds that
$$\frac{t}{a_T(t)-b_T(t_2^\pm(t))}=\frac{1\pm\dot b_T(\tilde t^\pm)}{\frac{a_T(t)}t-\frac{b_T(t)}{t}}$$
for some $\tilde t^\pm$ between $t$ and $t_2^\pm(t):=t_2^\pm(a_T(t),b_T(t),t)$.
\end{lemma}
\begin{proof}
 Rearranging terms and using definition~\eqref{t+-} of $t_2^\pm$, we get
\begin{align*}
 \frac{t}{a_T(t)-b_T(t_2^\pm(t))}=&\frac{t}{a_T(t)- b_T(t)}\left(1+\frac{b_T(t_2^\pm(t))-b_T(t)}{a_T(t)-b_T(t_2^\pm(t))}\right) \\
 =& \frac{t}{a_T(t)- b_T(t)}\left(1\pm\frac{b_T(t_2^\pm(t))-b_T(t)}{t_2^\pm(t)-t}\right)\,.
\end{align*}
Applying the mean value theorem to the term in brackets, the claim follows.
\end{proof}

Finally, we have all necessary ingredient for Part II of the proof:

\begin{proof}[Proof of Lemma~\ref{lem:2}, estimate~\eqref{Orte}]
    It suffices to provide the bound
    \eqref{eq:Aest}
    on terms $A_1,\, A_2,\, A_3,$ $A_4,\, A_5$ given in
    \eqref{A1}-\eqref{A5}. \\

 \emph{Term $A_5$}: Starting from the expressions
 $$\frac{\partial }{\partial u}\left[(1-u^2)^{\frac32}\frac{1\mp v}{u-v}\right] = \frac{(1-u^2)^{\frac12}(1\mp v)(-2u^2+3uv-1)}{(u-v)^2}$$
 and
 $$\frac{\partial }{\partial v}\left[(1-u^2)^{\frac32}\frac{1\mp v}{u-v}\right] = \frac{(1-u^2)^{\frac32}(1\mp u)}{(u-v)^2}$$
 for $u\neq v\in]-1,1[$ and employing formula~\eqref{tpunkt} for $\dot t_2^\pm$,
     we get
 \begin{align*}
  &\frac{\text d }{\text dr}\left[(1-\dot a_T(r)^2)^{\frac32}\frac{1\mp \dot b_T(t_2^\pm(r))}{\dot a_T(r)-\dot b_T(t_2^\pm(r))}\right] \\
  = & \ddot a_T(r)\frac{(1-a_T(r)^2)^{\frac12}(1\mp \dot b_T(t_2^\pm(r)))(-2\dot a_T(r)^2+3\dot a_T(r)\dot b_T(t_2^\pm(r))-1)}{(\dot a_T(r)-\dot b_T(t_2^\pm(r)))^2} \\
  &+\ddot b_T(t_2^\pm(r))\frac{(1-\dot a_T(r)^2)^{\frac52}}{(1\pm\dot b_T(t_2^\pm(r)))(\dot a_T(r)-\dot b_T(t_2^\pm(r)))^2}
 \end{align*}
for $T\le r\le t^*$.
The velocity estimate~\eqref{V} and Lemma~\ref{lem:apunkt} applied to $\dot
a_T(r)-\dot b_T(t_2^\pm(r))$ imply
\begin{equation}\label{a5betrag}
 \left|\frac{\text d }{\text dr}\left[(1-\dot a_T(r)^2)^{\frac32}\frac{1\mp \dot b_T(t_2^\pm(r))}{\dot a_T(r)-\dot b_T(t_2^\pm(r))}\right]\right|\le C\left(\ddot a_T(r)-\ddot b_T(t_2^\pm(r))\right).
\end{equation}
From the FST equation \eqref{WF}, eq.~\eqref{Beschleunigung} for $\ddot a_T$ and
Lemma~\ref{lem:t+-}, we obtain
\begin{equation}\label{a5a}
 \ddot a_T(r)\le\frac{C}{\left(a_T(r)-b_T(r)\right)^2}
\end{equation}
for $T\le r\le t^*$ as well as
\begin{equation*}
 -\ddot b_T(t_2^\pm(r))\le C\left[\frac1{\left(a_T(r)-b_T(r)\right)^2} +
 \frac1{\left(a_T(t_2^\pm(r))-b_T(t_2^\pm(r))\right)^2}\right]\,,
\end{equation*}
for $T\le r\le t^*$ in the ``$t_2^-(r)$'' case and for $T\le r\le t^*$, but only such that 
$t_2^-(r)> T^+$ holds, in the ``$t_2^+(r)$'' case; recall that for particle $b_T$ the
FST equations \eqref{WF} are only guaranteed to hold from $T^+$ on by Theorem
\ref{conditional}.
By the mean value theorem, the velocity estimate, definition~\eqref{t+-} of
$t_2^\pm$ and Lemma~\ref{lem:t+-}, we furthermore find
\begin{equation*}
\begin{split}
 a_T(t_2^+(r))-b_T(t_2^+(r))\ge& a_T(r)-V(t_2^+(r)-r)-b(t_2^+(r)) \\
 =& a_T(r)-V(a_T(r)-b_T(t_2^ +(r)))-b_T(t_2^+(r)) \ge\frac C{a_T(r)-b_T(r)}
 \end{split}
\end{equation*}
and
\begin{equation*}
a_T(t_2^-(r))-b_T(t_2^-(r))\ge\frac C{a_T(r)-b_T(r)}
 \end{equation*}
holds true for the values of $r$ considered above. This is because assumption~\eqref{Streuung}
implies that $a_T(r)-b_T(r)$ is monotonically decreasing in $r$.
If, instead, $T\le r\le t^*$ but now $t_2^-(r)\le T^+$, then, using in addition
definition~\eqref{Asymptoten} of the asymptotes, we get
\begin{equation*}
 -\ddot b_T(t_2^-(r))=\frac{\eta_2}{t_2^-(r)^2}<\frac{\eta_2}{\left(r-t_2^-(r)\right)^2}=\frac{\eta_2}{\left(a_T(r)-b_T(t_2^-(r))\right)^2}\le\frac C{\left(a_T(r)-b_T(r)\right)^2}\,.
\end{equation*}
In summary, in all cases, we find
\begin{equation}\label{a5b}
 -\ddot b_T(t_2^\pm(r))\le\frac C{\left(a_T(r)-b_T(r)\right)^2}\,.
\end{equation}
Using this for $b_T$ together with estimate~\eqref{a5a} for $\ddot a_T$ in
\eqref{a5betrag}, we can provide the claimed bound \eqref{eq:Aest} on $A_5$ in \eqref{A5}  by
using once more Lemma~\ref{lem:t+-} and assumption~\eqref{Streuung}. Indeed, we
infer the even better bound
\begin{equation*}
 \begin{split}
  |A_5(t)|\le \int_T^t\int_T^s\frac{C}{\left(a_T(r)-b_T(r)\right)^3}\text dr\text ds
  \le \int_T^t\int_T^s\frac{C}{|t|^3}\text dr\text ds\le\frac C{|t|}\,.
 \end{split}
\end{equation*}

\emph{Term $A_4$:}
We compute the derivatives in the integrand of $A_4$ in \eqref{A4} via
formulas~\eqref{lemeta2} and \eqref{lemeta3}. Using the assumption~\eqref{Streuung}
to estimate the denominator, the boundedness $|\dot a_T|, |\dot b_T|\le1$ 
to estimate the remaining velocities as well as
and estimates~\eqref{a5a} and \eqref{a5b} for $\ddot a_T$ and $\ddot b_T$, gives
\begin{align*}
 &\left|\frac{\text d}{\text ds}\left(\left(1-\dot a_T(s)^2\right)^{\frac32}\frac{1-\dot b_T(t_2^\pm(s))^2}{\left(\dot a_T(s)-\dot b_T(t_2^\pm(s))\right)^2}\right)\right| \\
 =&\left|\ddot a_T(s)\frac{\left(1-\dot a_T(s)^2\right)^{\frac12}\left(1-\dot b_T(t_2^\pm)^2\right)\left(-\dot a_T(s)^2+3\dot a_T(s)\dot b_T(t_2^\pm(s))-2\right)}{\left(\dot a_T(s)-\dot b_T(t_2^\pm(s))\right)^3}\right.  \\
 &\left.+ \ddot b_T(t_2^\pm(s))\frac{2(1-\dot a_T(s))\left(1-\dot a_T(s)^2\right)^{\frac32}\left(1-\dot a_T(s)\dot b_T(t_2^\pm(s))\right)}{\left(1-\dot b_T(t_2^\pm(s))\left(\dot a_T(s)-\dot b_T(t_2^\pm(s))\right)^3\right)}\right| \\
 \le& C\left(\ddot a_T(s)-\ddot b_T(s)\right) \le\frac{C}{\left(a_T(s)-b_T(s)\right)^2}
\end{align*}
for $T\leq s\leq t^*$.
For sufficiently large negative $s$, assumption~\eqref{Streuung} implies
\begin{equation}\label{a4a}
\frac{a_T(s)-b_T(s)}{2}\ge\frac{\mu s}2\ge1\,.
\end{equation}
Invoking Lemmata~\ref{lem:t+-} and \ref{lem:4}, yields
\begin{equation}\label{ln}
 0=\ln(1)\le\ln\left(a_T(s)-b_T(t_2^\pm(s))\right)\le\ln\left(\frac{a_T(s)-b_T(s)}{1-V}\right)
\end{equation}
and therefore,
\begin{equation*}
 |A_4(t)|\le \int_T^t\frac{C}{(a_T(s)-b_T(s))^2}\ln\left(\frac{a_T(s)-b_T(s)}{1-V}\right)\text ds\,.
\end{equation*}
The integrand is increasing in $a_T(s)-b_T(s)$ for sufficiently large negative $s$, so, using \eqref{a4a}, $A_4$ can further be bounded by
\begin{align*}
 |A_4(t)|\le\int_T^t\frac{C\ln(C|s|)}{s^2}\text ds\le C\int_T^t\frac{\ln(C|s|)-1}{s^2}\text ds = C\int_T^t\frac{\text d}{\text ds}\frac{\ln(C|s|)}{-s}\text ds\le C\frac{\ln|t|}{|t|}\,.
\end{align*}

So far the terms had sufficient decay in time to asymptotically vanish. As
discussed above, in the
following estimates it will be important to
observe certain cancellations 
between the term in order to provide the corresponding estimates
\eqref{eq:Aest}.\\

\emph{Term $A_1$}:
Since $T\le t<0$ we have $|\frac{t-T}{T}|\le1$, and furthermore
\begin{align}
 &|A_1(t)| \\
    \label{sum1}
 \le&\left|\frac{\kappa_a}2\left(1-\dot a_T(T)^2\right)^{\frac32}\left[\frac{1-\dot b_T(t_2^-(T))^2}{\left(\dot a_T(T)-\dot b_T(t_2^-(T))\right)^2} 
 +\frac{1-\dot b_T(t_2^+(T))^2}{\left(\dot a_T(T)-\dot b_T(t_2^+(T))\right)^2}\right]-\eta_1\right| \\
 &+\left|\frac{\kappa_a}2\left(1-\dot a_T(T)^2\right)^{\frac32}\frac{1+\dot b_T(t_2^-(T))}{\dot a_T(T)-\dot b_T(t_2^-(T))}\left(\frac{T}{a_T(T)-b_T(t_2^-(T))}-\frac{1-\dot b_T(t_2^-(T))}{\dot a_T(T)-\dot b_T(t_2^-(T))}\right)\right| \\
    \label{sum2}
 &+\left|\frac{\kappa_a}2\left(1-\dot a_T(T)^2\right)^{\frac32}\frac{1-\dot b_T(t_2^+(T))}{\dot a_T(T)-\dot b_T(t_2^+(T))}\left(\frac{T}{a_T(T)-b_T(t_2^+(T))}-\frac{1+\dot b_T(t_2^+(T))}{\dot a_T(T)-\dot b_T(t_2^+(T))}\right)\right| \\
    \label{sum3}
 \le&\frac C{|T|} + C\left|\frac{T}{a_T(T)-b_T(t_2^-(T))}-\frac{1-\dot b_T(t_2^-(T))}{\dot a_T(T)-\dot b_T(t_2^-(T))}\right| \\
 &+ C\left|\frac{T}{a_T(T)-b_T(t_2^+(T))}-\frac{1+\dot b_T(t_2^+(T))}{\dot a_T(T)-\dot b_T(t_2^+(T))}\right|\,,
\end{align}
where Lemma~\ref{lem:eta} has been used for the summand \eqref{sum1} 
and the usual velocity estimates for the other two summands
\eqref{sum2},\eqref{sum3}.
With $\tilde T^\pm$ between $T$ and $t_2^\pm(T)$ in order to apply
Lemma~\ref{lem:letztes}, and in addition
using Lemma~\ref{lem:apunkt} and $\ddot b_T<0$, we find
\begin{align*}
 &\left|\frac{T}{a_T(T)-b_T(t_2^\pm(T))}-\frac{1\pm\dot b_T(t_2^\pm(T))}{\dot a_T(T)-\dot b_T(t_2^\pm(T))}\right| \\
 \le& \left|\frac{1\pm\dot b_T(\tilde T^\pm)}{\frac{a_T(T)}T-\frac{b_T(T)}T}-\frac{1\pm\dot b_T(\tilde T^\pm)}{\dot a_T(T)-\dot b_T(t_2^\pm(T))}\right|
 + \left|\frac{\dot b_T(\tilde T^\pm)-\dot b_T(t_2^\pm(T))}{\dot a_T(T)-\dot b_T(t_2^\pm(T))}\right| \\
 \le& 2\left|\frac{1}{\frac{a_T(T)}T-\frac{b_T(T)}T}-\frac{1}{\dot a_T(T)-\dot b_T(t_2^\pm(T))}\right| + \frac2\mu\left|\dot b_T(T)-\dot b_T(t_2^\pm(T))\right|\,.
\end{align*}
By assumption~\eqref{Streuung} and Lemma~\ref{lem:apunkt}, the denominators in
the first modulus are at most $\frac\mu2$. Applying the mean value theorem,
Lemma~\ref{lem:apunkt}, the fact that $a_T(t)=x(t)$ and $b_T(t)=y(t)$ for $t\le
T$ and the definition of the asymptotes in \eqref{Asymptoten}, yields
\begin{align*}
 &\left|\frac{T}{a_T(T)-b_T(t_2^\pm(T))}-\frac{1\pm\dot b_T(t_2^\pm(T))}{\dot a_T(T)-\dot b_T(t_2^\pm(T))}\right| \\
 \le& \frac8{\mu^2} \left|\frac{a_T(T)}T-\frac{b_T(T)}T - \left(\dot a_T(T)-\dot b_T(t_2^\pm(T))\right)\right| + \frac2\mu\left|\dot b_T(T)-\dot b_T(t_2^\pm(T))\right| \\
 \le& C\left|\frac{x(T)}T-\dot x(T)\right| + C\left|\frac{y(T)}T - \dot y(T)\right| + C\left|\dot b_T(T)-\dot b_T(t_2^\pm(T))\right|\\
 \le& C\frac{\ln|T|}{|T|}\le C\frac{\ln|t|}{|t|},
\end{align*}
and therefore, $$|A_1(t)|\le C\frac{\ln|t|}{|t|}\,,$$
which complies with the claimed estimate \eqref{eq:Aest}.

\emph{Term $A_3,\,A_2$:} We compute for term $A_2$ in \eqref{A2} for $T\leq t$
\begin{align}\label{a21}
\begin{split}
 |A_2(t)|\le&
 \left|\left[\frac{\kappa_a}2\left(1-\dot a_T(t)^2\right)^{\frac32}\frac{1-\dot b_T(t_2^-(t))^2}{\left(\dot a_T(t)-\dot b_T(t_2^-(t))\right)^2}
 -\frac{\eta_1}2\right]\ln\left(a_T(t)-b_T(t_2^-(t))\right)\right| \\
 &+\left|\left[\frac{\kappa_a}2\left(1-\dot a_T(t)^2\right)^{\frac32}\frac{1-\dot b_T(t_2^+(t))^2}{\left(\dot a_T(t)-\dot b_T(t_2^+(t))\right)^2}
 -\frac{\eta_1}2\right]\ln\left(a_T(t)-b_T(t_2^+(t))\right)\right| \\
 &+\frac{\eta_1}2\left|\ln\left(-\frac{a_T(t)-b_T(t_2^-(t))}{t}\right)-\ln\left(-\frac{u_{-\infty}-v_{-\infty}}{1-v_{-\infty}}\right)\right| \\
 &+\frac{\eta_1}2\left|\ln\left(-\frac{a_T(t)-b_T(t_2^+(t))}{t}\right)-\ln\left(-\frac{u_{-\infty}-v_{-\infty}}{1+v_{-\infty}}\right)\right| \\
 \le& \frac C{|T|}\left|\ln\left(\frac{a_T(t)-b_T(t)}{1-V}\right)\right| \\
 &+ C\left(\left|\frac{a_T(t)-b_T(t_2^-(t))}{t} - \frac{u_{-\infty}-v_{-\infty}}{1-v_{-\infty}}\right| 
 + \left|\frac{a_T(t)-b_T(t_2^+(t))}{t} - \frac{u_{-\infty}-v_{-\infty}}{1+v_{-\infty}}\right|\right),
 \end{split}
\end{align}
where in the first two summands Lemma~\ref{lem:eta} has been applied to the
factors in the brackets involving the 
difference for $\frac{\eta_1}2$, and the logarithmic factors
were bounded by estimate~\eqref{ln}. For the differences of the logarithms
we exploited once again the mean value theorem noting that their arguments
are all bounded from below by $\frac\mu2$ thanks to assumption~\eqref{Streuung}
and $|v_{-\infty}|<1$. With $\tilde t^\pm$ between $t$ and $t_2^\pm(t)$ in order
to apply 
Lemma~\ref{lem:letztes} and noting that $\ddot b_T<0$, we find
\begin{align}\label{a22}
\begin{split}
 &\left|\frac{a_T(t)-b_T(t_2^\pm(t))}{t} - \frac{u_{-\infty}-v_{-\infty}}{1\pm v_{-\infty}}\right| \\
 \le&  \left|\frac{\frac{a_T(t)}t-\frac{b_T(t)}t-u_{-\infty}+v_{-\infty}}{1\pm\dot b_T(\tilde t^\pm)}\right|
 +\left|\frac{u_{-\infty}-v_{-\infty}}{1\pm\dot b_T(\tilde t^\pm)} - \frac{u_{-\infty}-v_{-\infty}}{1\pm v_{-\infty}}\right| \\
 \le& C\left(\left|\frac{a_T(t)}t-u_{-\infty}\right| + \left|\frac{b_T(t)}t-v_{-\infty}\right| + \left|\dot b_T(\tilde t^\pm)-v_{-\infty} \right|\right)\,.
 \end{split}
\end{align}
Lemmata~\ref{lem:apunkt} and \ref{lem:1} and the fact that
$\ddot b_T<0$ implies
\begin{align*}
\left|\dot b_T(\tilde t^\pm)-v_{-\infty} \right|\le&\left|\dot b_T(t_2^+(t))-v_{-\infty} \right|\le\left|\dot b_T(t_2^+(t))-\dot b_T(t)\right|+\left|\dot b_T(t)-v_{-\infty} \right|
\le\frac C{|t|}\,.
\end{align*}
Collecting these results going back to eq.~\eqref{a22} and then \eqref{a21}, we
have shown
\begin{align}\label{a23}
 |A_2(t)|\le \frac C{|t|}\left(\left|\ln|a_T(t)-b_T(t)|\right|+\left|a_T(t)-u_{-\infty}t\right| + \left|b_T(t)-v_{-\infty}t\right| + 1 \right).
\end{align}
Note that the term $A_3(t)$ in \eqref{A3} is constant in $t$ and fulfills
$A_3(t)=A_2(T)$. If $t=T$, then
$a_T(t)=x(t)$ and $b_T(t)=y(t)$ and, for sufficiently large negative $t$,
the definition of the asymptotes $(x,y)$ in \eqref{Asymptoten} yields
\begin{equation*}
 |A_3(t)|=|A_2(T)|\le C\frac{\ln|T|}{|T|}\le C\frac{\ln|t|}{|t|}\,.
\end{equation*}
In order to provide a similar bound for $A_2(t)$ for general $t$, we first need
an upper bound on $a_T(t)-b_T(t)$ to estimate the right-hand side of
\eqref{a23}. This can be obtained by collecting the preceding
estimates together in the form~\eqref{ASumme2}: First, we may omit the terms in
\eqref{ASumme2} involving $\ln\left(a_T(t)-b_T(t_2^\pm(t))\right)$, which are
negative by \eqref{ln}, and find
\begin{align*}
 a_T(t)\le
 x_{-\infty}+u_{-\infty}t+\frac{\eta_1}2\ln\left(\frac{(u_{-\infty}-v_{-\infty})^2}{1-v_{-\infty}^2}\right)+\frac{C}{\sqrt{|t|}}\le
 C_1+u_{-\infty}t+\frac{C}{\sqrt{|t|}}\,,
\end{align*}
and likewise
\begin{align*}
 b_T(t)\ge C_2+v_{-\infty}t-\frac{C}{\sqrt{|t|}},
\end{align*}
where $C_1, C_2$ are not necessarily positive constants.
Recalling the choice $u_{-\infty}<v_{-\infty}$, we have
\begin{equation*}
 0<a_T(t)-b_T(t)\le C + (u_{-\infty}-v_{-\infty})t+\frac{C}{\sqrt{|t|}}\le C\left(1+|t|\right)
\end{equation*}
and, for sufficiently large negative $t$,
\begin{align}\label{a24}
 \left|\ln\left(a_T(t)-b_T(t)\right)\right|\le C\ln|t|\,.
\end{align}
Applying now \eqref{ASumme2} in the opposite direction and using, beside the
estimate just derived, Lemma~\ref{lem:4} and inequality \eqref{Streuung} for the
velocities,  we obtain
\begin{align*}
 a_T(t)\ge C_3+u_{-\infty}t-C\ln|t|-\frac C{\sqrt{|t|}}\,.
\end{align*}
Together with the upper bound, this implies
\begin{equation*}
 \left|a_T(t)-u_{-\infty}t\right|\le C\ln|t|
\end{equation*}
for sufficiently large negative $t$. Going back to 
\eqref{a23}, the latter bound and the corresponding one for $b_T$
together with estimate~\eqref{a24}, we find
$$|A_2(t)|\le C\frac{\ln|t|}{|t|}.$$
which again complies with the necessary estimate \eqref{eq:Aest}.\\

In summary, we have shown that \eqref{eq:Aest} holds for $T\leq t\leq t^*$
which concludes the proof of Lemma~\ref{lem:2}.
\end{proof}

\vskip1cm

\textbf{Acknowledgement.} This work was partially funded by the Elite Network of Bavaria
through the Junior Research Group `Interaction between Light and Matter'.

\bibliographystyle{plain}

\vskip1cm

\end{document}